\documentclass[11pt]{amsart}
\usepackage{mathtools} \mathtoolsset{showonlyrefs,showmanualtags}
\usepackage[
    colorlinks=true,       
    linkcolor=magenta,          
    citecolor=magenta,        
    filecolor=magenta,      
    urlcolor=magenta,           
    pagebackref=true]{hyperref}

\usepackage{stmaryrd,mathrsfs}
\usepackage{xcolor}
\numberwithin{equation}{section}
\allowdisplaybreaks[4]
\usepackage{amsmath,amsthm,amscd,amssymb,amsfonts, amsbsy}
\usepackage{latexsym}
\usepackage{exscale}
\newtheorem{Theorem}[equation]{Theorem}

\def\XXint#1#2#3{{\setbox0=\hbox{$#1{#2#3}{\int}$}
\vcenter{\hbox{$#2#3$}}\kern-.5\wd0}}

\def\bbR{\mathbb{R}}

\begin{document}

\title[Sharp weighted estimates involving one supremum]{Sharp weighted estimates involving one supremum}

\author{Kangwei Li}

\address{BCAM, Basque Center for Applied Mathematics, Mazarredo, 14. 48009 Bilbao Basque Country, Spain}
\email{kli@bcamath.org}

\thanks{The author is supported by Juan de la Cierva-Formaci\'on 2015 FJCI-2015-24547, the Basque Government through the BERC 2014-2017 program and
 Spanish Ministry of Economy and Competitiveness MINECO: BCAM Severo Ochoa excellence accreditation SEV-2013-0323.
}

\date{\today}


\begin{abstract}
In this note, we study the sharp weighted estimate involving one supremum. In particular, we give a positive answer to an open question raised by Lerner and Moen \cite{LM}. We also extend the result to rough homogeneous singular integral operators.
\end{abstract}

\maketitle

\section{Introduction and main result}
Our main object is the following so-called sparse operator:
\[
A_{\mathcal S} (f)(x)= \sum_{Q\in \mathcal S} \langle f \rangle_Q \chi^{}_Q(x),
\]
where $\mathcal S\subset \mathcal D$ is a sparse family, i.e., for all (dyadic) cubes $Q\in \mathcal S$, there exists $E_Q\subset Q$ which are pairwise disjoint and $|E_Q|\ge \gamma |Q|$ with $0<\gamma<1$, and $\langle f \rangle_Q=\frac 1{|Q|}\int_Q f$.  We only consider the sparse operator because  it dominates Calder\'on-Zygmund operator pointwisely, see \cite{CR, LN, Lac15, HRT, Ler2016}.

We are going to study the sharp weighted bounds of $A_{\mathcal S} $. Before that, let us recall
\begin{align*}
[w]_{A_p}&=\sup_Q A_p(w, Q):=\sup_Q \langle w \rangle_Q \langle w^{1-p'} \rangle_Q^{p-1},\\
 [w]_{A_\infty}&= \sup_Q A_\infty(w, Q):=\sup_Q\frac{\langle M(w\chi_Q)\rangle_Q}{\langle w \rangle_Q}.
\end{align*}
In \cite{HL}, Hyt\"onen and Lacey proved the following estimate:
\begin{equation}\label{eq:hl}
\|A_{\mathcal S} \|_{L^p(w)}\le c_n [w]_{A_p}^{\frac 1p}([w]_{A_\infty}^{\frac 1{p'}}+[w^{1-p'}]_{A_\infty}^{\frac 1{p}}),
\end{equation}
which generalizes the famous $A_2$ theorem, obtained by Hyt\"onen in \cite{H2012} (We also remark that when $p=2$, \eqref{eq:hl} was obtained by Hyt\"onen and P\'erez in \cite{HP}).
It was also conjectured in \cite{HL} that
\begin{equation}
\|A_{\mathcal S} \|_{L^p(w)}\le c_n ([w]_{A_p^{\frac 1p}A_\infty^{\frac 1{p'}}}+[w^{1-p'}]_{A_p^{\frac 1{p'}}A_\infty^{\frac 1{p}}}),
\end{equation}
where
\[
[w]_{A_p^{\alpha}A_r^{\beta}}:=\sup_Q A_p(w, Q)^\alpha A_r(w, Q)^\beta.
\]
This conjecture, which is usually referred as the one supremum conjecture, is still open.  Before this conjecture was formulated, Lerner \cite{Ler2013} obtained the following mixed $A_p$-$A_r$ estimate:
\[
\|A_{\mathcal S} \|_{L^p(w)}\le c_{n,p,r}( [w]_{A_p^{\frac 1{p-1}}A_r^{1-\frac 1{p-1}}}+[w^{1-p'}]_{A_{p'}^{\frac 1{p'-1}}A_r^{1-\frac 1{p'-1}}}),
\]
which was further extended by Lerner and Moen \cite{LM} to the $r=\infty$ case with Hrus\v c\v ev $A_\infty$ constant:
\[
\|A_{\mathcal S} \|_{L^p(w)}\le c_{n,p}( [w]_{A_p^{\frac 1{p-1}}(A_\infty^{\exp})^{1-\frac 1{p-1}}}+[w^{1-p'}]_{A_{p'}^{\frac 1{p'-1}}(A_\infty^{\exp})^{1-\frac 1{p'-1}}}),
\]
where $A_\infty^{\exp}(w, Q)=\langle w \rangle_Q \exp(\langle \log w^{-1} \rangle_Q)$.  Some further extension can be also found in \cite{Li15}. Comparing this result with the one supremum conjecture, besides replacing the Fujii-Wilson $A_\infty$ constant by Hrus\v c\v ev $A_\infty$ constant, the main difference is that the power of $A_p$ constant is larger, leading to a result which is weaker than the one supremum conjecture. However, there is also another idea, which is replacing $A_p$ by $A_q$, where $q<p$. Our main result follows from this idea and it is formulated as follows\begin{Theorem}\label{thm:main}
Let $1\le q<p$ and $w\in A_q$. Then
\[
\|A_{\mathcal S} \|_{L^p(w)}\le c_{n,p, q}[w]_{A_q^{\frac 1{p}}(A_\infty^{\exp})^{\frac 1{p'}}}.
\]
\end{Theorem}
This result was conjectured by Lerner and Moen, see \cite[p.251]{LM}. It improves the previous result of Duoandikoetxea \cite{Duo}, i.e.,
\[
\|A_{\mathcal S} \|_{L^p(w)}\le c_{n,p, q}[w]_{A_q},
\] proved by means of extrapolation. In the next section, we will give a proof for this theorem. Extensions to rough homogeneous singular integrals will be provided in Section \ref{sec:rough}.

\section{Proof of Theorem \ref{thm:main}}
Before we state our proof, we would like to demonstrate our understanding of this $A_q$ condition, which allows us to avoid extrapolation or interpolation completely. We can rewrite the $A_q$ condition in the following form:
\begin{align*}
\langle w \rangle_Q \langle w^{1-q'} \rangle_Q^{q-1}&= \langle w \rangle_Q \langle w^{(1-p')\frac{p-1}{q-1}} \rangle_Q^{q-1}\\
&:=\langle w \rangle_Q \langle \sigma^{\frac 1{p'}} \rangle_{\bar A, Q}^p,
\end{align*}
where $\bar A(t)= t^{p'(p-1)/(q-1)}=t^{\frac p{q-1}}$ and as usual, $\sigma=w^{1-p'}$. So we have seen that $A_q$ condition is actually power bumped $A_p$ condition! Now we are ready to present our proof. Without loss of generality, we can assume $f\ge 0$. By duality, we have
\begin{align*}
\|A_{\mathcal S} (f)\|_{L^p(w)}&=\sup_{\|g\|_{L^{p'}(w)}=1}\int A_{\mathcal S} (f) g w\\
&= \sup_{\|g\|_{L^{p'}(w)}=1}  \sum_{Q\in \mathcal S} \langle f \rangle_Q \langle g \rangle_Q^w w(Q)\\
&\le \sup_{\|g\|_{L^{p'}(w)}=1}  \sum_{Q\in \mathcal S} \langle f w^{\frac 1p}\rangle_{A,Q} \langle w^{-\frac 1p}\rangle_{\bar A, Q}\langle g \rangle_Q^w \langle w\rangle_Q |Q|\\
&\times \exp(\langle \log w^{-1}\rangle_Q)^{\frac 1{p'}}\exp(\langle \log w\rangle_Q)^{\frac 1{p'}}\\
&\le [w]_{A_q^{\frac 1{p}}(A_\infty^{\exp})^{\frac 1{p'}}}\sup_{\|g\|_{L^{p'}(w)}=1} \Big( \sum_{Q\in \mathcal S} \langle f w^{\frac 1p}\rangle_{A,Q}^p |Q| \Big)^{\frac 1p}\\
&\times \Big( \sum_{Q\in \mathcal S} (\langle g \rangle_Q^w)^{p'} \exp(\langle \log w\rangle_Q)|Q| \Big)^{\frac 1{p'}}\\
&\le c_n\gamma^{-1} p\|M_A\|_{L^p}[w]_{A_q^{\frac 1{p}}(A_\infty^{\exp})^{\frac 1{p'}}}\|f\|_{L^p(w)},
\end{align*}
where in the last step, we have used the sparsity and the Carleson embedding theorem.

\section{Rough homogeneous singular integral operators}\label{sec:rough}
Recall that the rough homogeneous singular integral operator $T_\Omega$ is defined by
\[
T_\Omega(f)(x)=p.v. \int_{\bbR^n}\frac{\Omega(y')}{|y|^n}f(x-y) dy,
\]
where $\int_{S^{n-1}} \Omega=0$. The quantitative weighted bound of $T_\Omega$ with $\Omega \in L^\infty$ has been studied in \cite{HRT}, based on refinement of the ideas in \cite{DR}, see also a recent paper by the author and P\'erez, Rivera-R\'ios and Roncal \cite{LPRR}, relying upon the sparse domination formula established in \cite{CCDO}.

Our main result in this section is stated as follows.
\begin{Theorem}\label{thm:rough}
Let $1\le q<p$, $w\in A_q$ and $\Omega\in L^\infty(S^{n-1})$. Then
\[
\|T_\Omega \|_{L^p(w)}\le c_{n,p, q}[w]_{A_q^{\frac 1{p}}(A_\infty^{\exp})^{\frac 1{p'}}}.
\]
\end{Theorem}

\begin{proof}
The proof is again based on the sparse domination formula \cite{CCDO} (see also a very recent paper by Lerner \cite{Ler2017}). It suffices to prove
\[
\|A_{r,\mathcal S} \|_{L^p(w)}\le c_{n,p, r,q}[w]_{A_q^{\frac 1{p}}(A_\infty^{\exp})^{\frac 1{p'}}},
\]
where $1<r<\frac{p}{q}$ and
\[
A_{r,\mathcal S} (f)=\sum_{Q\in \mathcal S} \langle |f|^r \rangle_Q^{\frac 1 r} \chi_Q.
\]
Denote $\bar B(t)=t^{\frac{p'(p-1)}{r(q-1)}}=t^{\frac p{r(q-1)}}$. Again, we assume $f\ge 0$. By duality, we have
\begin{align*}
\|A_{r,\mathcal S} (f)\|_{L^p(w)}&=\sup_{\|g\|_{L^{p'}(w)}=1}\int A_{r, \mathcal S} (f) g w\\
&= \sup_{\|g\|_{L^{p'}(w)}=1}  \sum_{Q\in \mathcal S} \langle f^r \rangle_Q^{\frac 1r} \langle g \rangle_Q^w w(Q)\\
&\le \sup_{\|g\|_{L^{p'}(w)}=1}  \sum_{Q\in \mathcal S} \langle f^r w^{\frac r p}\rangle_{B,Q}^{\frac 1r} \langle w^{-\frac r p}\rangle_{\bar B, Q}^{\frac 1r}\langle g \rangle_Q^w \langle w\rangle_Q |Q|\\
&\times \exp(\langle \log w^{-1}\rangle_Q)^{\frac 1{p'}}\exp(\langle \log w\rangle_Q)^{\frac 1{p'}}\\
&\le [w]_{A_q^{\frac 1{p}}(A_\infty^{\exp})^{\frac 1{p'}}}\sup_{\|g\|_{L^{p'}(w)}=1} \Big( \sum_{Q\in \mathcal S} \langle f^r w^{\frac rp}\rangle_{B,Q}^\frac pr |Q| \Big)^{\frac 1p}\\
&\times \Big( \sum_{Q\in \mathcal S} (\langle g \rangle_Q^w)^{p'} \exp(\langle \log w\rangle_Q)|Q| \Big)^{\frac 1{p'}}\\
&\le c_n \gamma^{-1}p\|M_B\|_{L^{p/r}}^{\frac 1r}[w]_{A_q^{\frac 1{p}}(A_\infty^{\exp})^{\frac 1{p'}}}\|f\|_{L^p(w)},
\end{align*}
where again, in the last step we have used the sparsity and the Carleson embedding theorem.
\end{proof}


\begin{thebibliography}{EEE}
\bibitem{CCDO}
Jos\'e M. Conde-Alonso, Amalia Culiuc, Francesco Di Plinio, and Yumeng Ou, A sparse domination principle for rough singular integrals, Anal. \& PDE, 10(2017), no.5, 1255-1284.

\bibitem{CR}
Jos\'e M. Conde-Alonso and Guillermo Rey,
 A  pointwise  estimate  for  positive dyadic shifts and some applications, Math. Ann. 365(2016), 1111--1135.

 \bibitem{Duo}
 Javier Duoandikoetxea. Extrapolation of weights revisited: new
proofs and
sharp bounds.
J. Funct. Anal., 260(2011):1886--1901.

\bibitem{DR}
Javier Duoandikoetxea and Jos\'e  L. Rubio de Francia, Maximal and
singular integral operators via
Fourier transform estimates,
Invent. Math.
84
(1986), 541--561.

\bibitem{H2012}
Tuomas P. Hyt\"onen,
The sharp weighted bound for general Calder\'on-Zygmund operators. Ann. of Math. (2) 175 (2012), no. 3, 1473--1506.

\bibitem{HL}
Tuomas P. Hyt\"onen, Michael T. Lacey,  The $A_p$-$A_\infty$ inequality for general Calder\'on-Zygmund operators. Indiana Univ. Math. J. 61 (2012), no. 6, 2041--2092.

\bibitem{HP}
Tuomas P. Hyt\"onen and Carlos P\'erez, Sharp weighted bounds involving $A_\infty$, Anal. \& PDE, 6 (2013), no. 4, 777--818.

\bibitem{HRT}
Tuomas P. Hyt\"onen, Luz Roncal and Olli Tapiola, Quantitative weighted estimates for rough homogeneous singular integrals,  Israel J. Math. 218(2017), 133--164.


\bibitem{Lac15}
 Michael T. Lacey,  An  elementary  proof  of  the $A_2$ bound, Israel J. Math. 217(2017), 181--195.


\bibitem{Ler2013}
Andrei K. Lerner, Mixed $A_p$-$A_r$ inequalities for classical singular integrals and Littlewood-Paley operators. J. Geom. Anal. 23 (2013), no. 3, 1343--1354.

\bibitem{Ler2016}
Andrei K. Lerner, On pointwise estimates involving sparse operators, New York J. Math. 22(2016) 341--349.

\bibitem{Ler2017}
Andrei K. Lerner, A weak type estimate for rough singular integrals, preprint, available at  arxiv:1705.07397.

\bibitem{LM}
Andrei K. Lerner, Kabe Moen,
Mixed $A_p$-$A_\infty$ estimates with one supremum.
Studia Math. 219 (2013), no. 3, 247--267.

\bibitem{LN}
Andrei K. Lerner, Fedor Nazarov, Intuitive  dyadic  calculus:   the  basics.
Preprint, 2015. available at {arXiv:1508.05639}.


\bibitem{Li15}
Kangwei Li, Two weight inequalities for bilinear forms, Collect. Math., 68(2017),129--144.

\bibitem{LPRR}
Kangwei Li, Carlos P\'erez, Israel P. Rivera-R\'ios and Luz Roncal, Weighted norm inequalities for rough
singular integral operators, preprint, available at arXiv:1701.05170.

\end{thebibliography}
\end{document}